\documentclass[12pt]{article}
\usepackage[a4paper, total={7in, 10in}]{geometry}
\usepackage{latexsym, amsmath, amssymb}
\usepackage{graphicx}
\usepackage{epstopdf}
\usepackage{amsmath, amsthm, amscd, amsfonts}
\usepackage{amsmath, xcolor}
\usepackage{array}
\usepackage{multirow}
\usepackage{stfloats}
\usepackage{amssymb}
\usepackage{ltablex}
\usepackage{float}
\usepackage{cite}
\usepackage{enumerate}
\usepackage{placeins}
\usepackage{comment}
\usepackage{soul}
\usepackage{color}
\usepackage{lscape}
\usepackage{soul}
\usepackage{mathtools}
\usepackage{caption} 
 
\usepackage[colorlinks,urlcolor=blue]{hyperref}
\newtheorem{theorem}{Theorem}[section]

\newtheorem{lemma}{Lemma}[section]

\newtheorem{definition}{Definition}[section]
\newtheorem{example}{Example}[section]

\begin{document}
	
	\begin{center}\Large
		\textbf{Reversible and Reversible-Complement Double Cyclic Codes over $\mathbb{F}_4+v\mathbb{F}_4$ and its Application to DNA Codes}
	\end{center}

    
	\begin{center}
		Divya Acharya$^{1}$, Prasanna Poojary$^{1}$, Vadiraja Bhatta G R$^{2}$
		\end{center}
	
	\begin{center}
		$^1$Department of Mathematics, Manipal Institute of Technology Bengaluru, Manipal Academy of Higher Education, Manipal, India\\ 
		$^2$ Department of Mathematics, Manipal Institute of Technology, Manipal Academy of Higher Education, Manipal, India\\ 
		{\it \textcolor{blue}{acharyadivya1998@gmail.com; }}{\it \textcolor{blue}{poojary.prasanna@manipal.edu;}}\\
		{\it \textcolor{blue}{poojaryprasanna34@gmail.com;  vadiraja.bhatta@manipal.edu}}\\
		\end{center}
	
	\abstract{ In this article, we study the algebraic structure of double cyclic codes of length $(m, n)$ over $\mathbb{F}_4$ and we give a necessary and sufficient condition for a double cyclic code over $\mathbb{F}_4$ to be reversible. Also, we determine the algebraic structure of double cyclic codes of length $(m, n)$ over $\mathbb{F}_4+v\mathbb{F}_4$ with $v^2=v$, satisfying the reverse constraint and the reverse-complement constraint. Then we establish a one-to-one correspondence $\psi$ between the 16 DNA double pairs $S_{D_{16}} $ and the 16 elements of the finite ring $\mathbb{F}_4+v\mathbb{F}_4$.
     We also discuss the GC-content of DNA double cyclic codes.
     }\\
     
	

	
	\section{Introduction}
    Cyclic codes are significant linear codes recognized for their algebraic structure and practical usability. While cyclic codes over finite fields have been extensively studied, attention towards cyclic codes over finite rings increased after Hammons et al.\cite{hammons1994z}, which revealed good non-linear binary codes as Gray images of linear codes over $\mathbb{Z}_4$. This work spurred further investigation into the structure of cyclic codes across various finite rings.

    Borges et al. \cite{Borges2018double} investigated double cyclic codes over ${\mathbb{Z}}_2$. Then, double cyclic codes over $\mathbb{Z}_2+u\mathbb{Z}_2$ are studied in \cite{aydogdu2024double}. Research interest in the algebraic structure of double cyclic codes over various finite fields and finite rings has increased significantly. Then the double cyclic code over $\mathbb{F}_4+v\mathbb{F}_4$ with $v^2=v$ is studied  by Bathala and Seneviratne \cite{bathala2021some}.
    
	Deoxyribonucleic acid (DNA) is an essential component of all living things as the genetic information carrier. It is composed of strands that are connected and twisted into the shape of a double helix. Each strand is a sequence made up of four distinct nucleotides: two purines: Adenine (A) and Guanine (G), and two pyrimidines: Thymine (T ) and Cytosine (C). The ends of a single DNA strand are chemically polar with the so-called $5^\prime$ end and the $3^\prime$ end, which implies that the strands are oriented. Base pairing is the process by which two strands of DNA bind to produce a double strand. The strands are linked following the Watson–Crick model. According to this model, each purine base is paired with a unique pyrimidine base and vice versa. In particular, every A is paired with a T and every C is paired with a G and vice versa. The paired nucleotides are called Watson–Crick complement(WCC) of each other. If we denote the WCC of $x$ as $\overline{x}$ then we have $\overline{A} = T$, $ \overline{T} = A$, $\overline{C} = G$, $\overline{G} = C$. DNA strand pairing occurs in the opposite direction and in the reverse order. For instance, the Watson–Crick complementary (WCC) strand of $3^\prime$-ATTCGGC-$5^\prime$ is the strand $5^\prime$-GCCGAAT-$3^\prime$.
    
    In 1994, Adleman \cite{adleman1994molecular} started DNA computing. Adleman used DNA molecules to solve an example of an NP-complete problem in his work. In order to study DNA computing using the algebraic coding theory techniques, researchers studied error-correcting codes over finite fields and finite rings of cardinality $4^n$ by mapping the DNA nucleotides to the elements of finite fields and finite rings. Algebraic theory and application have relied heavily on the widely studied class of algebraic codes known as cyclic codes. Following that, studies on cyclic DNA codes and their generalizations advanced quickly. Cyclic DNA codes over finite rings and fields play a key role in DNA computation. 

    Gaborit and King constructed DNA codes over GF(4) \cite{gaborit2005linear} and Abualrub et al. \cite{abualrub2006construction}, in 2005 and 2006, respectively. Later, Siap et al. \cite{siap2009cyclic} studied DNA codes over the finite ring $\mathbb{F}_2[u]/\langle u^2 -1 \rangle $ with four elements. Guenda et al.\cite{guenda2013construction} constructed DNA codes over $\mathbb{F}_2+u\mathbb{F}_2$ with $u^2= 0$. 
    
    Yildiz and Siap in \cite{yildiz2012cyclic} studied DNA codes over the ring $\mathbb{F}_2[u]/\langle u^4 -1 \rangle$ with 16 elements. This article investigates the algebraic properties of cyclic DNA codes and matches sixteen elements of this ring with a set of paired DNA nucleotides. Later, Bayram et al. covered certain DNA applications of linear codes over the ring $\mathbb{F}_4+v\mathbb{F}_4$ $(v^2=v)$ with 16 elements in \cite{bayram2016codes}. Several authors have also made considerable use of cyclic codes in the construction of DNA codes \cite{NabilBennenni, dinh2018cyclic, dinh2018dna, liu2020dna}. Recently, Kanlaya and Klin-Eam \cite{kanlaya2023constructing} studied DNA double cyclic codes of length $(\alpha,  \beta)$ over $\mathbb{F}_2+u\mathbb{F}_2$ $( u^2 = 0)$, where $\alpha$ and $\beta$ are odd positive integers, and in \cite{kanlaya2025constructing}, they extended this work to $\alpha$ and $\beta$ are odd  and even positive integers respectively. Motivated by these works, we examine DNA codes over the finite ring of 16 elements, namely, $\mathbb{F}_4+v\mathbb{F}_4$ with $v^2=v$ in this paper.

    The article is organized as follows: Section \ref{Sec2} includes some basic background. In Section \ref{Sec3}, we give some definitions and the algebraic structure of the double cyclic code over $\mathbb{F}_4$. We provide the necessary and sufficient conditions for the double cyclic codes over $\mathbb{F}_4$ to be reversible. Section \ref{Sec4} provides the algebraic structure of the double cyclic code over $\mathbb{F}_4+v\mathbb{F}_4$ with $v^2=v$. In Section \ref{Sec5}, we first establish a one-to-one
    correspondence $\psi$ between the 16 DNA double pairs $S_{D_{16}}= \{AA, AT, AG, AC, TT, T A, TG, TC, GG, GA, GC, GT, CC, CA, CG, CT \} $ and the 16 elements of the finite ring $\mathbb{F}_4+v\mathbb{F}_4$. Also, we give a map $\theta $ from double cyclic codes over $\mathbb{F}_4+v\mathbb{F}_4$ of length $(m,n)$ to
    $S^{2m}_{D_{4}}\times S^{2n}_{D_{4}}$, where $m$ and $n$ are two odd positive integers. Then we study double cyclic codes over $\mathbb{F}_4+v\mathbb{F}_4$ satisfying the reverse constraint and reverse complement
    constraint. In Section \ref{Sec6} we study the GC content of double cyclic codes over $\mathbb{F}_4+v\mathbb{F}_4$. 

	\section{Preliminaries}\label{Sec2}
	Let us consider the finite field with four element $ \mathbb{F}_4 = \{0, 1,\gamma, \gamma + 1\}$, where
	$\gamma^2 = \gamma + 1$. Let $\mathcal{R}$ denote the commutative ring $\mathbb{F}_4+v\mathbb{F}_4=\{a+vb: a,b\in \mathbb{F}_4\}$ with $v^2=v$. $\mathcal{R}$ is a finite non-chain ring with characteristic two and $\vert \mathcal{R} \vert=16$. The invertible elements in $\mathcal{R}$ are $\{1,\gamma,\gamma+1, v+\gamma, 1+v+\gamma, 1+v\gamma, 1+v+v\gamma, 1+\gamma+v\gamma, v+\gamma+v\gamma\}$
	and the non-invertible elements are $\{0, v, v\gamma, 1+v, v+v\gamma,\gamma+v\gamma, 1+v+\gamma+v\gamma\}$. Further, $\mathcal{R}$ is a semilocal Frobenius ring with two maximal ideals $\langle v \rangle$ and  $\langle v+1 \rangle$.
	
	A code $\mathcal{C}$ of length $n$ over $\mathcal{R}$ is a nonempty subset of $\mathcal{R}^n$. An element of $\mathcal{C}$ is called a codeword. $\mathcal{C}$ is called a  linear code over $\mathcal{R}$ if $\mathcal{C}$ is an $\mathcal{R}$-submodule of $\mathcal{R}^n$. 
	In the study of cyclic codes, it is useful to identify a vector  $v = (v_0, v_1, \ldots,  v_{n-1})$ in $\mathcal{R}^n$ with the polynomial  $v(x) = v_ 0+ v_1x +\cdots+v_{n-1}x^{n-1}$. It is well-known that a linear code $\mathcal{C}$ of length $n$ over $\mathcal{R}$ is an cyclic if and only if $\mathcal{C}$ is an ideal of $\frac{\Re[x]}{\langle x^n-1 \rangle}$.	The Hamming weight of $c$, represented as $wt_H(c)$, is the nonzero entries of a codeword $c$. The Hamming distance $d_H(c,c^{\prime})$, of two
	words $c$ and $c^{\prime}$, equals the number of components in which they differ. i.e., $d_H(c,c^{\prime})=wt_H(c-c^{\prime})$. 
	The Hamming distance of code $\mathcal{C}$ contains at least two words, $d_H(\mathcal{C}) = min\{d(c, c^{\prime}): c,c^{\prime} \in \mathcal{C}, c \neq c^{\prime}\}$.
	
	There is a mapping known as the Gray map that converts linear codes over $\mathcal{R}$ to linear codes over $ \mathbb{F}_4$, which is a crucial characteristic of codes over the ring $\mathcal{R}$. For any $a+vb\in \mathcal{R}$ the Gray map from $\mathcal{R}$ to $\mathbb{F}^2_4$ is defined by
	\begin{align}\label{eqn6}
	\notag	\phi:\mathcal{R}& \longrightarrow \mathbb{F}^2_4\\
		a+vb&\longmapsto (a+b, a).
	\end{align}
	
	For any polynomial $f (x) = a_0 + a_1x +\cdots+ a_kx^k \in \mathcal{R}[x]$,  where $a_k\neq 0$,  the reciprocal polynomial of $f (x)$ is defined as
	\begin{equation*}
		f (x)^* =x^kf(x^{-1})= a_k + a_{k-1}x +\cdots+a_0x^k. 
	\end{equation*}


	\section{$\mathbb{F}_4$-double cyclic code}\label{Sec3}
	Let $\mathcal{C}$ be a binary code of length $\ell$. Let $m$ and $n$ be non-negative integers such that $\ell = m +n$.
	We consider a partition of the set of the $n$ coordinates into two subsets of $m$ and $n$ coordinates
	respectively, so that $\mathcal{C}$ is a subset of $\mathbb{F}^m_4\times \mathbb{F}^n_4$.
	
	\begin{definition}
		Let $\mathcal{C}$ be a binary linear code of length $\ell = m + n$. The code $\mathcal{C}$ is called
		$\mathbb{F}_4$-double cyclic code if
		\begin{align*}
			&(\alpha_{0}, \alpha_{1},\ldots,  \alpha_{m-1}\vert \beta_{0}, \beta_{1},\ldots,  \beta_{n-1})\in \mathcal{C}\\
			\implies&
			(\alpha_{m-1}, \alpha_{0},\ldots,  \alpha_{m-2}\vert \beta_{n-1}, \beta_{0},\ldots,  \beta_{n-2})\in \mathcal{C}.
		\end{align*}

	\end{definition}
	Let $\gamma=(\alpha_{0}, \alpha_{1},\ldots,  \alpha_{m-1}\vert \beta_{0}, \beta_{1},\ldots, \beta_{n-1})$ be a codeword in $\mathcal{C}$ and let $i$ be an integer. Then the $i^{th}$ shift of $\boldsymbol{\delta}$ is given by
	\begin{equation*}
		\boldsymbol{\delta}^{(i)}=(\alpha_{0-i}, \alpha_{1-i},\ldots,  	\alpha_{m-1-i}\vert \beta_{0-i}, \beta_{1-i},\ldots,  \beta_{n-1-i}),
	\end{equation*}
	where the subscripts are read modulo $m$ and $n$, respectively.

	Let $\mathcal{C} \subseteq \mathbb{F}^m_4\times \mathbb{F}^n_4$ be a $\mathbb{F}_4$-double cyclic code. Let $\mathcal{C}_m$ be the canonical projection of $\mathcal{C}$
	on the first $m$ coordinates and $\mathcal{C}_n$ on the last $n$ coordinates. Note that $\mathcal{C}_m$ and $\mathcal{C}_n$ are binary
	cyclic codes of length $m$ and $n$, respectively. The code $\mathcal{C}$ is called separable if it is the direct
	product of $\mathcal{C}_m$ and $\mathcal{C}_n$.
	
	There is a bijective map between $\mathbb{F}^m_4\times \mathbb{F}^n_4$ and $\frac{\mathbb{F}_4[x]}{\langle x^m-1 \rangle }\times \frac{\mathbb{F}_4[x]}{\langle x^n-1\rangle }$ given
	by:
	\begin{equation*}
		(\alpha_{0}, \alpha_{1},\ldots,  \alpha_{m-1}\vert \beta_{0}, \beta_{1},\ldots,  \beta_{n-1})\mapsto (\alpha_{0}+ \alpha_{1}x+\cdots+ \alpha_{m-1}x^{m-1}\vert \beta_{0}+ \beta_{1}x+\cdots+  \beta_{n-1}x^{n-1}).
	\end{equation*}
	We denote the image of the vector $\boldsymbol{\delta}$ by $\boldsymbol{\delta}(x)$.

	\begin{definition}
		Let $\mathcal{A}_{m,n}$ denote the ring $\frac{\mathbb{F}_4[x]}{\langle x^m-1 \rangle }\times \frac{\mathbb{F}_4[x]}{\langle x^n-1\rangle }$. The operation 
		\begin{equation*}
			\star:\mathbb{F}_4[x]\times \mathcal{A}_{m,n} \longrightarrow \mathcal{A}_{m,n}
			\end{equation*}
		defined as 
		\begin{equation*}
				\lambda (x) \star (p(x) \vert q(x)) = (\lambda(x)p(x) \vert \lambda(x)q(x))
		\end{equation*}
		where $\lambda (x) \in \mathbb{F}_4[x]$ and $(p(x) \vert q(x)) \in \mathcal{A}_{m,n}$.
	\end{definition}
	The ring $\mathcal{A}_{m,n}$ with the external operation $\star$ is a $\mathbb{F}_4[x]$-module. 
	
 	Let $\boldsymbol{\delta}(x)=(\alpha(x) \vert \beta(x))$
	be an element of $\mathcal{A}_{m,n}$. Note that if we operate $\boldsymbol{\delta}(x)$ by $x$ we get
	\begin{align*}
		x\star \boldsymbol{\delta}(x)&=x \star (\alpha(x) \vert \beta(x))\\
		&=x \star(\alpha_{0}+ \alpha_{1}x+\cdots+ \alpha_{m-1}x^{m-1}\vert \beta_{0}+ \beta_{1}x+\cdots+  \beta_{n-1}x^{n-1})\\
		&=(\alpha_{0}x+ \alpha_{1}x^2+\cdots+\alpha_{m-2}x^{m-1}+ \alpha_{m-1}x^{m}\vert \beta_{0}x+ \beta_{1}x^2+\cdots+  \beta_{n-2}x^{n-1}+\beta_{n-1}x^{n})\\
		&=(\alpha_{m-1}x^{m}+\alpha_{0}x+ \alpha_{1}x^2+\cdots+\alpha_{m-2}x^{m-1} \vert \beta_{n-1}x^{n}+\beta_{0}x+ \beta_{1}x^2+\cdots+  \beta_{n-2}x^{n-1}).
	\end{align*}
	Hence, $x\star \boldsymbol{\delta}(x)$ is the image of the vector $\boldsymbol{\delta}^{(1)}$. Thus, the operation of $\boldsymbol{\delta}(x)$ by $x$ in $\mathcal{A}_{m,n}$
	corresponds to a shift of $\boldsymbol{\delta}$. In general, $x^i  \boldsymbol{\delta}(x) = \boldsymbol{\delta}^{(i)}(x)$ for all $i$.

	The structure of $\mathbb{F}_4$-double cyclic code of length $(m, n)$ is given in \cite{bathala2021some}.

	\begin{theorem}\cite{bathala2021some}\label{thm1}
		Let $\mathcal{C}$ be an $\mathbb{F}_4$-double cyclic code of length $(m,n)$. Then $\mathcal{C} = \Big\langle(b(x) \vert 0), (l(x) \vert a(x)\Big\rangle$, where $l(x) \in F_4[x]$, $b(x)\vert(x^r-1)$ and $a(x)\vert(x^s-1)$. Moreover,
		\begin{enumerate}
			\item $deg(l(x)) <deg(b(x))$
			\item $b(x)$ divides $l(x)\frac{x^s-1}{a(x)}$
			\item If $deg(b(x)) = t_1 $ and $deg(a(x)) = t_2$, then $\mathcal{C}$ is spanned by $S_1 \cup S_2$, where $S_1 =\bigcup\limits_{i=0}^{r-t_1-1}x^i (b(x) \vert 0)$ and $S_2 =\bigcup\limits_{i=0}^{s-t_2-1}x^i (l(x) \vert a(x))$. Further, $\vert C\vert = 4^{r+s-t_1-t_2 }$.
		\end{enumerate}
	\end{theorem}
		\begin{lemma}\label{lemma2}
		Let $\mathcal{C} = \Big\langle(b(x) \vert 0), (l(x) \vert a(x)\Big\rangle $ be a $\mathbb{F}_4$-double cyclic code. Assume the
		generator polynomials of $\mathcal{C}$ satisfy the conditions in Proposition 1. Define the sets
		\begin{align*}
			\mathcal{S}_1&=\{(b(x)\vert 0), x\star (b(x)\vert 0), \ldots x^{m-deg(b(x))-1}\star (b(x)\vert 0) \}\\
			\mathcal{S}_2&=\{(l(x)\vert a(x)), x\star (l(x)\vert a(x)), \ldots x^{n-deg(a(x))-1}\star (l(x)\vert a(x)) \}.
		\end{align*}
		Then, $\mathcal{S}_1\cup \mathcal{S}_2$ forms a minimal generating set for $\mathcal{C}$ as a $\mathbb{F}_4$-module.
	\end{lemma}	
	Proof is similar to proof of Proposition 3 in \cite{borges2018z}
	
	\subsection{Reversible $\mathbb{F}_4$-double cyclic code}
	\begin{definition}
		The reverse of a $\mathbb{F}_4$-double cyclic code $\mathcal{C}$, denoted by $\mathcal{C}^r$, is a
		linear code over $\mathbb{F}_4$ defined as
		\begin{equation*}
			\mathcal{C}^r=\{(\alpha_{m-1}, \alpha_{m-2},\ldots,  \alpha_{0}\vert \beta_{n-1}, \beta_{n-2},\ldots,  \beta_{0}): (\alpha_{0}, \alpha_{1},\ldots,  \alpha_{m-1}\vert \beta_{0}, \beta_{1},\ldots,  \beta_{n-1})\in \mathcal{C}\}.
		\end{equation*}
	\end{definition}
	\begin{theorem}\label{thm2}
		The reverse of  $\mathbb{F}_4$-double cyclic code is also $\mathbb{F}_4$-double cyclic code.
		\begin{proof}
			Let $\mathcal{C}$ be a  $\mathbb{F}_4$-double cyclic code and $\mathcal{C}^r$ be the reverse of  $\mathbb{F}_4$-double cyclic code. To show $\mathcal{C}^r$ is $\mathbb{F}_4$-double cyclic code it is enough to show $(\alpha_{0},\alpha_{m-1},\ldots, \alpha_{1}\vert \beta_{0},\beta_{n-1},\ldots, \beta_{1})\in \mathcal{C}^r$ for all 
			$(\alpha_{0}, \alpha_{1},\ldots,  \alpha_{m-1}\vert \beta_{0}, \beta_{1},\ldots,  \beta_{n-1})\in \mathcal{C}$. Let $(\alpha_{0}, \alpha_{1},\ldots, \alpha_{m-1}\vert \beta_{0}, \beta_{1},\ldots, \beta_{n-1})\in \mathcal{C}$. Then $(\alpha_{m-1}, \ldots, \alpha_{1}, \alpha_{0}\vert \beta_{n-1}, \ldots, \beta_{1}, \beta_{0})\in \mathcal{C}^r$. 
			
			Now 
			
			\begin{align*}
				&(\alpha_{m-1}+\cdots+ \alpha_{1}x^{m-2}+\alpha_{0}x^{m-1}\vert \beta_{n-1}+\cdots+\beta_{1}x^{n-2}+ \beta_{0}x^{n-1})\in \mathcal{C}^r\\
				\implies&(\alpha_{0}+\alpha_{1}x+\cdots+ \alpha_{m-1}x^{m-1}\vert \beta_{0}+\beta_{1}x+\cdots+ \beta_{n-1}x^{n-1})\in \mathcal{C}\\
				\implies&x^{lcm(m,n)-1}\star\Big(\alpha_{0}+\alpha_{1}x+\cdots+\alpha_{m-1}x^{m-1}\Big\vert \beta_{0}+\beta_{1}x+\cdots+ \beta_{n-1}x^{n-1}\Big)\in \mathcal{C}\\
				\implies&\Bigg(x^{ma_m-1}\Big(\alpha_{0}+\alpha_{1}x+\cdots+ \alpha_{m-1}x^{m-1}\Big)\Bigg\vert x^{na_n-1}\Big(\beta_{0}+\beta_{1}x+\cdots+\beta_{n-1}x^{n-1}\Big)\Bigg)\in \mathcal{C}\\
				&\qquad \qquad \qquad \Big[\text{where},~ lcm(m,n)=ma_m~ \text{and}~ lcm(m,n)=na_n~\text {for positive integers}~ a_m ~ {and}~a_n. \Big]\\
				\implies&\Bigg(x^{(a_m-1)m+m-1}\Big(\alpha_{0}+\alpha_{1}x+\cdots+ \alpha_{m-1}x^{m-1}\Big)\Bigg\vert x^{(a_n-1)n+n-1}\Big(\beta_{0}+\beta_{1}x+\cdots+ \beta_{n-1}x^{n-1}\Big)\Bigg)\in \mathcal{C}\\
				\implies&\Bigg(x^{m-1}\Big(\alpha_{0}+\alpha_{1}x+\cdots+ \alpha_{m-1}x^{m-1}\Big)\Bigg\vert x^{n-1}\Big(\beta_{0}+\beta_{1}x+\cdots \beta_{n-1}x^{n-1}\Big)\Bigg)\in \mathcal{C}\\
				\implies&\Big(\alpha_{1}+\alpha_{2}x+\cdots+ \alpha_{m-1}x^{m-2}+\alpha_{0}x^{m-1}\Big\vert \beta_{1}+\beta_{2}x+\cdots+ \beta_{n-1}x^{n-2}+\beta_{0}x^{n-1}\Big)\in \mathcal{C}\\
				\implies&\Big(\alpha_{0}+\alpha_{m-1}x+\cdots+ \alpha_{2}x^{m-2}+\alpha_{1}x^{m-1}\Big\vert \beta_{0}+\beta_{n-1}x+\cdots+ \beta_{2}x^{n-2}+\beta_{1}x^{n-1}\Big)\in \mathcal{C}^r.
			\end{align*}
			Hence $\mathcal{C}^r$ is $\mathbb{F}_4$-double cyclic code.
		\end{proof}
	\end{theorem}
	
	We observe that the $\mathbb{Z}_2$-double cyclic code and the $\mathbb{F}_4$-double cyclic code have comparable structures. The necessary and sufficient condition for a $\mathbb{Z}_2$-double cyclic code to be reversible is studied by \cite{patanker2023reversible}. Since the results in this context are a simple extension of the $\mathbb{Z}_2$-double cyclic code analysis, we do not include the proofs.
	\begin{theorem}\label{thm3}
		Let $\mathcal{C} = \Big\langle(b(x) \vert 0), (l(x) \vert a(x))\Big\rangle$ be an $\mathbb{F}_4$-double cyclic code such that its generators satisfy conditions (1) and (2) of Proposition 2.2. If
		$m- deg(l(x)) \geq n- deg(a(x))$, then $\mathcal{C} $ is reversible if and only if the following
		conditions are satisfied
		\begin{enumerate}
			\item$ b(x) = b^\ast (x)$,
			\item $ a(x) = a^\ast(x)$,
			\item$ b(x)\Big\vert \Big(x^{m-n+deg(a(x))-deg(l(x))}l^\ast(x) -l(x)\Big)$.
		\end{enumerate}
	\end{theorem}

	\section{$\mathcal{R}$-double cyclic code}\label{Sec4}
		\begin{definition}
		Let $\mathfrak{C}$ be a binary linear code of length $\ell = m + n$. The code $\mathfrak{C}$ is called
		$\mathcal{R}$-double cyclic code if
		\begin{align*}
			&(\alpha_{0}, \alpha_{1},\ldots,  \alpha_{m-1}\vert \beta_{0}, \beta_{1},\ldots,  \beta_{n-1})\in \mathfrak{C}\\
			\implies&
			(\alpha_{m-1}, \alpha_{0},\ldots,  \alpha_{m-2}\vert \beta_{n-1}, \beta_{0},\ldots,  \beta_{n-2})\in \mathfrak{C}.
		\end{align*}

	\end{definition}

	Let $\mathfrak{C} \subseteq \mathcal{R}^m\times \mathcal{R}^n_4$ be a $\mathcal{R}$-double cyclic code. Let $\mathfrak{C}_m$ be the canonical projection of $\mathfrak{C}$
	on the first $m$ coordinates and $\mathfrak{C}_n$ on the last $n$ coordinates. Note that $\mathfrak{C}_m$ and $\mathfrak{C}_n$ are binary
	cyclic codes of length $m$ and $n$, respectively. The code $\mathfrak{C}$ is called separable if it is the direct
	product of $\mathfrak{C}_m$ and $\mathfrak{C}_n$.
	
	There is a bijective map between $\mathcal{R}^m\times \mathcal{R}^n$ and $\frac{\mathcal{R}[x]}{\langle x^m-1 \rangle }\times \frac{\mathcal{R}[x]}{\langle x^n-1\rangle }$ given
	by:
	\begin{equation*}
		(\alpha_{0}, \alpha_{1},\ldots,  \alpha_{m-1}\vert \beta_{0}, \beta_{1},\ldots,  \beta_{n-1})\mapsto (\alpha_{0}+ \alpha_{1}x+\cdots+ \alpha_{m-1}x^{m-1}\vert \beta_{0}+ \beta_{1}x+\cdots+  \beta_{n-1}x^{n-1}).
	\end{equation*}
	We denote the image of the vector $\boldsymbol{\delta}$ by $\boldsymbol{\delta}(x)$.

	The structure of $\mathcal{R}$-double cyclic code of length $(m, n)$ is given in \cite{bathala2021some}.
	
	\begin{theorem}\cite{bathala2021some}\label{thm4}
		Let $\mathfrak{C}$ be an R-double cyclic code of length $(m, n)$. Then
		\begin{equation*}
			\mathfrak{C} = \Big\langle(1 + v)b_1(x) + vb_2(x) \vert 0), ((1 + v)l_1(x) + vl_2(x)\Big\vert (1 + v)a_1(x) + va_2(x)\Big\rangle,
		\end{equation*}
		where $a_1(x)$, $a_2(x)$, $b_1(x)$ and $b_2(x)$ are monic polynomials such that $a_1(x)\vert(x^n-1)$,
		$a_2(x)\vert (x^n - 1)$, $b_1(x)\vert(x^m- 1)$ and $b_2(x)\vert(x^m- 1)$.
	\end{theorem}
	
	In \cite{bathala2021some}, it is shown that  $\mathcal{R}$-double cyclic code of length $(m, n)$ can be written as a direct sum of  $\mathbb{F}_4$-double cyclic codes of length $(m, n)$. The following theorems give the structure of $\mathcal{R}$-double cyclic code of length $(m, n)$ using the structure of $\mathbb{F}_4$-double cyclic codes of length $(m, n)$. Note that the structure of $\mathbb{F}_4$-double cyclic codes of length $(m, n)$ is given in Theorem \ref{thm1}.

	\begin{theorem}\cite{bathala2021some}\label{thm5}
		Let $\mathfrak{C}=(1+v)\mathcal{C}_1\oplus v\mathcal{C}_2$ be a linear code of length $m+n$ over $\mathcal{R}$. Then $\mathcal{C}$ is an $\mathcal{R}$-double cyclic code of length $(m,n)$ if and only if $\mathcal{C}_1$ and $\mathcal{C}_2$ are $\mathbb{F}_4$-double cyclic codes of length $(m,n)$.
	\end{theorem}

	\begin{theorem}\cite{bathala2021some}\label{thm6}
		Let $\mathcal{C}_1 = \Big\langle(b_1(x) \vert 0), (l_1(x) \vert a_1(x))\Big\rangle$ and $\mathcal{C}_2 = \Big\langle(b_2(x) \vert 0), (l_2(x) \vert a_2(x))\Big\rangle$
		be two $\mathbb{F}_4$-double cyclic codes of length $(m, n)$,  where $l_1(x), l_2(x) \in   \mathbb{F}_4[x]$, $ b_1(x)\vert (x^m- 1)$, $ b_2(x)\vert (x^m- 1)$, $a_1(x)\vert(x^n-1)$ and $a_2(x)\vert(x^n-1)$. If $\mathfrak{C}=(1+v)\mathcal{C}_1\oplus v\mathcal{C}_2$ is an $\mathcal{R}$-double cyclic code of length $(m, n)$, then 
		\begin{equation*}
			\mathfrak{C} = \Big \langle(1 + v)b_1(x) + vb_2(x) \vert 0), ((1 +
			v)l_1(x) + vl_2(x)\Big\vert (1 + v)a_1(x) + va_2(x)\Big \rangle.
		\end{equation*}
	\end{theorem}
	
	\section{DNA codes over $\mathcal{R}$}\label{Sec5}
	The set	$S_{D_{4}}= \{A, T, G, C\}$ called the DNA alphabet. DNA code is a set of codewords $(x_0, x_1,\ldots , x_{n-1})$, where
	$x_i \in \{A, T, G, C\}$. The elements $\{0, 1,\gamma, \gamma + 1\}$ of $\mathbb{F}_4$ are related with the DNA codons
	$S_{D_{4}}$ such that $A \rightarrow0$, $T \rightarrow
	1$, $C \rightarrow \gamma$,$ G \rightarrow \gamma + 1 = \gamma^2$.
	Next, we describe the set $S_{D_{16}}$, which comes from \cite{oztas2013lifted}
	\begin{equation*}
		S_{D_{16}}= \{AA, AT, AG, AC, TT, T A, TG, TC, GG, GA, GC, GT, CC, CA, CG, CT\}.
	\end{equation*}
	The set $S_{D_{16}}$ denotes the set of DNA double pairs. Now define a map $\psi$ from $\mathcal{R}$ to $S_{D_{16}}$, i.e.,
	\begin{equation*}
		\psi : \mathcal{R} \longrightarrow S_{D_{16}}, 
	\end{equation*}
	which gives one to one correspondence between elements of $\mathcal{R}$ and DNA double pairs of $S_{D_{16}}$, which are given in the Table \ref{DNA Double pair Table}. 
	
	\begin{table}[H]
		\centering
		\caption{DNA Double Pairs with elements of $\mathcal{R}$}
		\label{DNA Double pair Table}
		\begin{tabular}{|c|c|c|}
			\hline
			Elements of $\mathcal{R}$ &Gray Images &DNA Double Pairs  \\
			$a$&$\phi(a)$&$\psi(a)$\\
			\hline
			$0 $                    &   $(0,0)$                     &   AA\\
			$1  $                   &   $(1,1)$                     &   TT\\
			$\gamma$                &   $(\gamma, \gamma)$          &   CC\\
			$1 +\gamma $            &   $(1 + \gamma, 1 + \gamma)$  &   GG\\
			$v$                     &   $(1, 0)$                    &   TA\\
			$1 +v $                 &   $(0, 1)$                    &   AT\\
			$v +\gamma$             &   $(1 + \gamma, \gamma)$      &   GC\\
			$1 + v +\gamma $        &   $(\gamma, 1 + \gamma)$      &   CG\\
			$v\gamma$               &   $(\gamma, 0)$               &   CA\\
			$1 + v\gamma $          &   $(1 + \gamma, 1)$           &   GT\\
			$\gamma + v\gamma$      &   $(0,\gamma) $               &   AC\\
			$1 +\gamma+v\gamma$     &   $(1, 1 + \gamma)$           &   TG\\
			$v + v\gamma $          &   $(1 + \gamma, 0)$           &   GA\\
			$1 + v + v\gamma$       &   $(\gamma, 1)$               &   CT\\
			$\gamma+v+v\gamma$      &   $(1,\gamma)$                &   TC\\
			$1+\gamma+v+v\gamma$    &   $(0, 1 + \gamma)$           &   AG\\
			\hline
		\end{tabular}
	\end{table}
	\begin{definition}
		Let $\mathfrak{C}$ be a double cyclic
		code of length $(m,n)$ over $\mathcal{R}$ and $\big(\alpha_{0},\ldots, \alpha_{m-1}\vert \beta_{0},\ldots, \beta_{n-1}\big) \in \mathfrak{C}$. Then we define the following map using Table \ref{DNA Double pair Table}. 
		\begin{align*}
			\theta : \mathfrak{C} &\longrightarrow S^{2m}_{D_{4}}\times S^{2n}_{D_{4}}\\ \\
			\big(\alpha_{0},\ldots, \alpha_{m-1}\vert \beta_{0},\ldots, \beta_{n-1}\big)&\longmapsto \big(\psi(\alpha_{0}),\ldots, \psi(\alpha_{m-1})\vert \psi(\beta_{0}),\ldots,\psi( \beta_{n-1})\big).
		\end{align*}
	\end{definition}
	i.e., for any codeword  $x$ in $\mathfrak{C}$, $\theta(x)$ is the corresponding DNA sequence in $\theta(\mathfrak{C})\subseteq S^{2m}_{D_{4}}\times S^{2n}_{D_{4}}$.
    
	The codons satisfy the Watson-Crick complement which is given by $\overline{A} = T$, $ \overline{T} = A$, $\overline{C} = G$, $\overline{G} = C$. For any $a\in \mathcal{R}$, the complement of $a$, denoted by $\overline{a}$, is defined as $\overline{a}=b$ if $\psi(b)=\overline{\psi(a)}$. The complement of elements of  $\mathcal{R}$ is given in Table \ref{Complement Table}.

	\begin{table}[H]
		\centering
		\caption{Complement of elements of $\mathcal{R}$}
		\label{Complement Table}
		\begin{tabular}{|c|c|}
			\hline
			Elements of $\mathcal{R}$ &Complement of elements of $\mathcal{R}$ \\
			$a$&$\overline{a}$\\
			\hline
			$0 $                    &   $1  $     \\
			$1  $                   &   $0 $   \\
			$\gamma$                &   $1 +\gamma $  \\
			$1 +\gamma $            &   $\gamma$\\
			$v$                     &   $1 +v $   \\
			$1 +v $                 &   $v$ \\
			$v +\gamma$             &   $1 + v +\gamma $ \\
			$1 + v +\gamma $        &   $v +\gamma$   \\
			$v\gamma$               &   $1 + v\gamma $  \\
			$1 + v\gamma $          &   $v\gamma$ \\
			$\gamma + v\gamma$      &   $1 + \gamma + v\gamma $\\
			$1 + \gamma + v\gamma $ &   $\gamma + v\gamma$\\
			$v + v\gamma $          &   $1 + v + v\gamma$\\
			$1 + v + v\gamma$       &   $v + v\gamma $ \\
			$\gamma+v+v\gamma$      &   $1+\gamma+v+v\gamma$ \\
			$1+\gamma+v+v\gamma$    &   $\gamma+v+v\gamma$ \\
			\hline
		\end{tabular}
	\end{table}
	
	Let $x =(x_0, x_1, \ldots, x_{n-1}) \in \mathcal{R}^n$ be a codeword. Then
	\begin{enumerate}
		\item the reverse of $x$ is $x^r=(x_{n-1}, x_{n-2}, \ldots, x_1, x_0)$,
		\item  the complement of $x$ is $\overline{x}=(\overline{x_0},\overline{x_1}, \ldots, \overline{x_{n-1}})$,
		\item the reverse-complement, (called the Watson–Crick complement (WCC)) is $\overline{x^{r}}=(\overline{x_{n-1}},\overline{x_{n-2}}, \ldots, \overline{x_1},\overline{x_0})$.
	\end{enumerate}  
	
	\begin{definition}
		Let $\mathfrak{C}$ be a linear code of length $n$ over $\mathcal{R}$ and $x \in \mathfrak{C}$. Then $\mathfrak{C}$ is said to be
		\begin{enumerate}
			\item reversible if $x^r\in \mathfrak{C}$ ,
			\item complement if $\overline{x}\in \mathfrak{C}$,
			\item reversible-complement if $\overline{{x^r}}\in \mathfrak{C}$.
		\end{enumerate}
		
	\end{definition}

	\begin{definition}
		Let $\mathfrak{C}$ be a double code of length $(m, n)$ over $\mathcal{R}$ and $x \in \mathfrak{C}$. Then $\mathfrak{C}$ or equivalently $\theta(\mathfrak{C})$is said to be
		\begin{enumerate}
			\item reversible double cyclic DNA code if ${\theta (x)}^r\in \theta(\mathfrak{C})$ ,
			\item double cyclic DNA code if $\overline{{\theta (x)}^r}\in \theta(\mathfrak{C})$.
		\end{enumerate}
	\end{definition}

	\subsection{Reversible $\mathcal{R}$-double cyclic code}
	Now we discuss the reversible $\mathcal{R}$-double cyclic code of length $(m,n)$.
	
	\begin{theorem}\label{thm7}
		Let $\mathfrak{C}=(1+v)\mathcal{C}_1\oplus v\mathcal{C}_2$ be a double cyclic code of length $(m,n)$ over $\mathcal{R}$. Then $\mathfrak{C}$ is reversible if and only if $\mathcal{C}_1$ and $\mathcal{C}_2$ are reversible codes over $\mathbb{F}_4$.
	\end{theorem}
	\begin{proof}
		Let $\mathfrak{C}$ is reversible $\mathcal{R}$-double cyclic code and $\textbf{u}=(\alpha_{0},\ldots, \alpha_{m-1}\vert \beta_{0},\ldots, \beta_{n-1})\in \mathfrak{C}$. Since $\mathfrak{C}=(1+v)\mathcal{C}_1\oplus   v\mathcal{C}_2$ there exits $\textbf{c}_1=(p_{0},\ldots, p_{m-1}\vert q_{0},\ldots, q_{n-1})\in \mathcal{C}_1$ and $\textbf{c}_2=(r_{0},\ldots, r_{m-1}\vert s_{0},\ldots, s_{n-1})\in \mathcal{C}_2$ such that $\textbf{u}=(1+v)\textbf{c}_1+v\textbf{c}_2$ and $\alpha_i=(v+1)p_i+vr_i$ and $\beta_j=(v+1)q_j+vs_j$, where $0\leq i\leq m-1$ and $0\leq j\leq n-1$. Then 
		\begin{align}\label{eqn1}
			\notag\textbf{u}^r&=(\alpha_{0},\ldots, \alpha_{m-1}\vert \beta_{0},\ldots, \beta_{n-1})^r\\
			\notag&=(\alpha_{m-1},\ldots, \alpha_{0}\vert \beta_{n-1},\ldots, \beta_{0})\\
			\notag & = \Big(\big((v+1)p_{m-1}+vr_{m-1}\big), \ldots,\big((v+1)p_0+vr_0\big) \Big\vert \big((v+1)q_{n-1}+vs_{n-1} \big),\ldots,\big((v+1)q_0+vs_0 \big) \Big)           \\
			\notag&=(v+1)\big(p_{m-1},\ldots, p_0 \big\vert q_{n-1},\ldots,q_0\big)+v\big( q_{m-1},\ldots, q_0 \big\vert s_{n-1},\ldots,s_0\big)\\
			\textbf{u}^r&=(1+v)\textbf{c}^r_1+v\textbf{c}^r_2 
		\end{align}
		Since $\mathfrak{C}$ is reversible $ \textbf{u}^r=(1+v)\textbf{c}^r_1+v\textbf{c}^r_2 \in \mathfrak{C}$. Then there exists $\textbf{a}_1\in \mathcal{C}_1$ and $\textbf{a}_2\in \mathcal{C}_2$ such that
		\begin{align*}
			\textbf{u}^r=(1+v)\textbf{c}^r_1+v\textbf{c}^r_2=(1+v)\textbf{a}_1+v\textbf{a}_2
		\end{align*}
		which implies that 
		\begin{align}\label{eqn2}
			\notag\Big[(1+v)\textbf{c}^r_1+v\textbf{c}^r_2\Big]-        \Big[(1+v)\textbf{a}_1+v\textbf{a}_2\Big]&=0\\
			(1+v) \Big[\textbf{c}^r_1-\textbf{a}_1 \Big]+v\Big[\textbf{c}^r_2-\textbf{a}_2\Big]&=0
		\end{align}
		From Equation \ref{eqn2}, $ \textbf{c}^r_1-\textbf{a}_1=0$ and $\textbf{c}^r_2-\textbf{a}_2=0$. Thus $\textbf{c}^r_1=\textbf{a}_1\in \mathcal{C}_1$ and $\textbf{c}^r_2=\textbf{a}_2\in \mathcal{C}_2$. Hence $\mathcal{C}_1$ and $\mathcal{C}_2$ are reversible codes over $\mathbb{F}_4$.
		
		Conversely, assume $\mathcal{C}_1$ and $\mathcal{C}_2$ are reversible codes over $\mathbb{F}_4$. From Equation \ref{eqn1}, we have $\textbf{u}^r=(1+v)\textbf{c}^r_1+v\textbf{c}^r_2$. Since $\mathcal{C}_1$ and $\mathcal{C}_2$ are reversible, $\textbf{c}^r_1\in \mathcal{C}_1$ and $\textbf{c}^r_2\in \mathcal{C}_2$. Thus $\textbf{u}^r=(1+v)\textbf{c}^r_1+v\textbf{c}^r_2\in \mathfrak{C}$. Hence $\mathfrak{C}$ is reversible code over $\mathcal{R}$.
	\end{proof}
	
	\begin{example}
		Let $\mathcal{C}_1=\langle (x+1\vert 0),(\gamma+1\vert 1)\rangle$ and $\mathcal{C}_2=\langle (x+1\vert 0),(\gamma \vert x+1)\rangle$ be $\mathbb{F}_4$-double cyclic codes of length $(2,2)$. Since $\mathcal{C}_1$ and $\mathcal{C}_2$ satisfies all three conditions of Theorem \ref{thm3},  $\mathcal{C}_1$ and $\mathcal{C}_2$ are reversible codes. Hence by Theorem \ref{thm7}, $\mathfrak{C} = \big \langle(1 + v)(x+1) + v(x+1) \vert 0), ((1 +
		v)(\gamma +1) + v\gamma\big\vert (1 + v) + v(x+1)\big \rangle$ is reversible code.
	\end{example}

	\subsection{Reverse-Complement $\mathcal{R}$-double cyclic code}
	Now we discuss the reversible-complement $\mathcal{R}$-double cyclic code of length $(m,n)$. First, by observing Table \ref{Complement Table}, the following useful Lemmas are provided.
	
	\begin{lemma}\label{lemma3}
		Let $a,b\in \mathcal{R}$. Then
		\begin{enumerate}
			\item $a+\overline{a}=1$
			\item $\overline{(a+b)}=\overline{a}+\overline{b}+1$.
		\end{enumerate}
	\end{lemma}

	\begin{lemma}\label{lemma4}
		Let $\textbf{a},\textbf{b}\in \mathcal{R}^m\times \mathcal{R}^n$. Then
		\begin{enumerate}
			\item $\textbf{a}+\overline{\textbf{a}}=(\underbrace{1,\ldots, 1}_{\text{$m$}}\vert \underbrace{1,\ldots, 1}_{\text{$n$}})$
			\item $\overline{(\textbf{a}+\textbf{b})}=\overline{\textbf{a}}+\overline{\textbf{b}}+(\underbrace{1,\ldots, 1}_{\text{$m$}}\vert \underbrace{1,\ldots, 1}_{\text{$n$}})$.
		\end{enumerate}
	\end{lemma}

	\begin{theorem}\label{thm8}
		Let $\mathfrak{C}=(1+v)\mathcal{C}_1\oplus v\mathcal{C}_2$ be a double cyclic code of length $(m,n)$ over $\mathcal{R}$. Then $\mathfrak{C}$ is a reversible-complement if and only if 
		\begin{enumerate}
			\item $\mathfrak{C}$ reversible code
			\item $(\underbrace{1,\ldots, 1}_{\text{$m$}}\vert \underbrace{1,\ldots, 1}_{\text{$n$}})\in \mathfrak{C}$.
		\end{enumerate}
		\begin{proof}
			Let $\mathfrak{C}$ be a reversible-complement double cyclic code of length $(m,n)$ over $\mathcal{R}$. Then $\overline{\textbf{u}^r}=(\overline{\alpha_{m-1}},\ldots, \overline{\alpha_{0}}\vert \overline{\beta_{n-1}},\ldots, \overline{\beta_{0}})\in \mathfrak{C}$ for all $\textbf{u}=(\alpha_{0},\ldots, \alpha_{m-1}\vert \beta_{0},\ldots, \beta_{n-1})\in \mathfrak{C}$. Since $(\underbrace{0,\ldots, 0}_{\text{$m$}}\vert \underbrace{0,\ldots, 0}_{\text{$n$}})\in \mathfrak{C}$ and its Watson-Crick Complement is also in $\mathfrak{C}$, we have $(\underbrace{1,\ldots, 1}_{\text{$m$}}\vert \underbrace{1,\ldots, 1}_{\text{$n$}})=(\underbrace{\overline{0},\ldots, \overline{0}}_{\text{$m$}}\vert \underbrace{\overline{0},\ldots, \overline{0}}_{\text{$n$}})\in \mathfrak{C}$. Let $\textbf{u}=(\alpha_{0},\ldots, \alpha_{m-1}\vert \beta_{0},\ldots, \beta_{n-1})\in \mathfrak{C}$. Then by Lemma \ref{lemma4}
			\begin{align*}
				\textbf{u}^r&=(\alpha_{m-1},\ldots, \alpha_{0}\vert \beta_{n-1},\ldots, \beta_{0})\\
				&=(\overline{\alpha_{m-1}},\ldots, \overline{\alpha_{0}}\vert \overline{\beta_{n-1}},\ldots, \overline{\beta_{0}})+(\underbrace{1,\ldots, 1}_{\text{$m$}}\vert \underbrace{1,\ldots, 1}_{\text{$n$}}) \in \mathfrak{C}.
			\end{align*}
			Conversely, assume  $\mathfrak{C}$ reversible code and $(\underbrace{1,\ldots, 1}_{\text{$m$}}\vert \underbrace{1,\ldots, 1}_{\text{$n$}})\in \mathfrak{C}$. Let $\textbf{w}=(\delta_{0},\ldots, \delta_{m-1}\vert \eta_{0},\ldots, \eta_{n-1})\in \mathfrak{C}$. Then $\textbf{w}^r=(\delta_{m-1},\ldots, \delta_{0}\vert \eta_{n-1},\ldots, \eta_{0})\in \mathfrak{C}$. By Lemma \ref{lemma4},
			\begin{align*}
				\overline{\textbf{w}^r}&=(\overline{\delta_{m-1}},\ldots, \overline{\delta_{0}}\vert \overline{\eta_{n-1}},\ldots, \overline{\eta_{0}})\\
				&=(\delta_{m-1},\ldots, \delta_{0}\vert \eta_{n-1},\ldots, \eta_{0})+(\underbrace{1,\ldots, 1}_{\text{$m$}}\vert \underbrace{1,\ldots, 1}_{\text{$n$}}) \in \mathfrak{C}
			\end{align*}
			Hence $\mathfrak{C}$ is reversible-complement.
			
		\end{proof}
	\end{theorem}

	\section{GC Weight}\label{Sec6}
	In molecular biology, GC-content (or guanine–cytosine content) is the percentage of nitrogenous bases on a DNA molecule that are either guanine(G) or cytosine(C). The chemical bonds between the WCC pairs are different where the G and C pair need three hydrogen bonds, while A and T pairs need two hydrogen bonds. The total energy of the DNA molecule depends on the number of A and T pairs and the number of C and G pairs. As DNA with high energy GC-content is more stable than DNA with low energy GC-content, it is always desirable in a DNA code to have all codewords with the same GC-content, so that they have similar melting temperatures. In this section we study the GC-content of a DNA cyclic code.

	Let $W_{GC}(\zeta)$ denote GC weight of $\zeta$, for any $\zeta \in \mathfrak{C}$ and $W_{GC}(\mathfrak{C})$ denote GC weight of $\mathfrak{C}$.  

	We first provide the definition of the Hamming weight enumerator from \cite{huffman2010fundamentals}.

	\begin{definition}
		Hamming weight enumerator, denoted by $W_{\mathfrak{C}}(x)$, of code $\mathfrak{C}$ is defined by
		\begin{equation*}
			W_{\mathfrak{C}}(x)=\sum_{\zeta\in \mathfrak{C}} 	x^{W_H(\zeta)}.
		\end{equation*}
	\end{definition}

	The Gray map $\phi$ from $\mathcal{R}$ to $\mathbb{F}_4$ defined in Equation \ref{eqn6} can be extended from $\mathcal{R}^{m+n}$ to $\mathbb{F}^{2m+2n}_4$. For any $\big((a_{1,0}+vb_{1,0}), \ldots, (a_{1,m-1}+vb_{1,m-1})\big\vert (a_{2,0}+vb_{2,0}), \ldots, (a_{2,n-1}+vb_{2,n-1})\big) \in \mathcal{R}^{m+n}$, extended Gray map $\Phi$ from  $\mathcal{R}^{m+n}$ to $\mathbb{F}^{2m+2n}_4$ is given by
 
	\begin{equation*}
	\Phi:\mathcal{R}^{m+n} \longrightarrow \mathbb{F}^{2m+2n}_4
	\end{equation*}
	\begin{align*}
		&\Big((a_{1,0}+vb_{1,0}), \ldots, (a_{1,m-1}+vb_{1,m-1})\Big\vert (a_{2,0}+vb_{2,0}), \ldots, (a_{2,n-1}+vb_{2,n-1})\Big)\\
		&\longmapsto \Big( \big(a_{1,0}+b_{1,0}, \ldots, 	a_{1,m-1}+b_{1,m-1}\big\vert a_{2,0}+b_{2,0}, \ldots, a_{2,n-1}+b_{2,n-1}\big),\big(a_{1,0}, \ldots, a_{1,m-1}\Big\vert a_{2,0}, \ldots, a_{2,n-1}\big)\Big).
	\end{align*}

	\begin{theorem}\label{thm9}
		Let $\mathfrak{C}$ be a double cyclic code of length $m+n$ over $\mathcal{R}$. Then GC weight of $\mathfrak{C}$ over $\mathcal{R}$ is given by
	\begin{equation*}
		W_{GC}(\mathfrak{C})=\sum_{\Phi(\zeta)\in \Phi(\mathfrak{C})} x^{W_H\big(\Phi(\zeta)\big)+W_H\Big(\Phi(\zeta)+\big(\underbrace{1,\ldots, 1}_{\text{$2m$}}\big\vert \underbrace{1,\ldots, 1}_{\text{$2n$}}\big)\Big)- (2m+2n)},
	\end{equation*}
	where $\Phi(\mathfrak{C})=\{\Phi(\zeta):\zeta \in \mathfrak{C}\}$.
	\end{theorem}
	
	\begin{proof}
		Let $\zeta\in \mathfrak{C}$. Then 
		\begin{align}\label{eqn3}
			W_{GC}(\zeta)=2m+2n-\big[\text{Number of 0's in } \Phi(\zeta)+ \text{Number of 1's in } \Phi(\zeta)\big]
		\end{align}
		We have 
		\begin{equation}\label{eqn4}
			\text{Number of 0's in } \Phi(\zeta)= 2m+2n-W_H\big(\Phi(\zeta)\big)
		\end{equation}
		and 
		\begin{equation}\label{eqn5}
			\text{Number of 1's in } \Phi(\zeta)= 2m+2n-W_H\big(\Phi(\zeta)+(\underbrace{1,\ldots, 1}_{\text{$2m$}}\vert \underbrace{1,\ldots, 1}_{\text{$2n$}})\big)
		\end{equation}
		From Equation \ref{eqn3}, Equation \ref{eqn4} and Equation \ref{eqn5},
		\begin{align*}
			W_{GC}(\zeta)&=2m+2n-\Bigg[2m+2n-W_H\big(\Phi(\zeta)\big)+ 2m+2n-W_H\Big(\Phi(\zeta)+(\underbrace{1,\ldots, 1}_{\text{$2m$}}\vert \underbrace{1,\ldots, 1}_{\text{$2n$}})\Big)\Bigg]\\
			&=W_H\big(\Phi(\zeta)\big)+W_H\Big(\Phi(\zeta)+(\underbrace{1,\ldots, 1}_{\text{$2m$}}\vert \underbrace{1,\ldots, 1}_{\text{$2n$}})\Big)- (2m+2n)
		\end{align*}
		Then 
		\begin{equation*}
			W_{GC}(\mathfrak{C})=\sum_{\Phi(\zeta)\in \Phi(\mathfrak{C})} 	x^{W_H\big(\Phi(\zeta)\big)+W_H\Big(\Phi(\zeta)+\big(\underbrace{1,\ldots, 1}_{\text{$2m$}}\big\vert \underbrace{1,\ldots, 1}_{\text{$2n$}}\big)\Big)- 	(2m+2n)},
		\end{equation*}
		where $\Phi(\mathfrak{C})=\{\Phi(\zeta):\zeta \in \mathfrak{C}\}$.
	\end{proof}
	



\begin{thebibliography}{00}
    
    \bibitem{hammons1994z} A. R. Hammons, P. V. Kumar, A. R. Calderbank, N. J. Sloane, and P. Sol´e, “The $\mathbb{Z}_4$-linearity of Kerdock, Preparata, Goethals, and related codes,” IEEE Transactions on Information Theory, vol. 40, no. 2, pp. 301–319, 1994.
    
    \bibitem{Borges2018double} J. Borges, C. Fern´andez-C´ordoba, and R. Ten-Valls, “$\mathbb{Z}_2$-double cyclic codes,” Designs, Codes and Cryptography, vol. 86, p. 463–479, 2018.

    \bibitem{aydogdu2024double} I. Aydogdu, “On double cyclic codes over $\mathbb{Z}_2+u\mathbb{Z}_2$,” Advances in Mathematics of Communications, vol. 9, pp. 11076–11091, 2024.

    \bibitem{bathala2021some} S. Bathala and P. Seneviratne, “Some results on $\mathbb{F}_4[v]$-double cyclic codes,” Computational and Applied Mathematics, vol. 40, no. 2, p. 64, 2021.

    \bibitem{adleman1994molecular} L. M. Adleman, “Molecular computation of solutions to combinatorial problems,” Science, vol. 266, no. 5187, pp. 1021–1024, 1994.

    \bibitem{gaborit2005linear} P. Gaborit and O. D. King, “Linear constructions for DNA codes,” Theoretical Computer Science, vol. 334, no. 1-3, pp. 99–113, 2005.

    \bibitem{abualrub2006construction} T. Abualrub, A. Ghrayeb, and X. N. Zeng, “Construction of cyclic codes over $GF(4)$ for DNA computing,” Journal of the Franklin Institute, vol. 343, no. 4-5, pp. 448–457, 2006.

    \bibitem{siap2009cyclic} I. Siap, T. Abualrub, and A. Ghrayeb, “Cyclic DNA codes over the ring $\mathbb{F}_2[u]/\langle u^2-1\rangle$ based on the deletion distance,” Journal of the Franklin Institute, vol. 346, no. 8, pp. 731–740, 2009.

    \bibitem{guenda2013construction} K. Guenda and T. A. Gulliver, “Construction of cyclic codes over $\mathbb{F}_2+ u\mathbb{F}_2$ for DNA computing,” Applicable Algebra in Engineering, Communication and Computing, vol. 24, no. 6, pp. 445–459, 2013.

    \bibitem{yildiz2012cyclic} B. Yildiz and I. Siap, “Cyclic codes over $\mathbb{F}_2[u]/\langle u^4-1\rangle$ and applications to DNA codes,” Computers \& Mathematics with Applications, vol. 63, no. 7, pp. 1169–1176, 2012.

    \bibitem{bayram2016codes} A. Bayram, E. S. Oztas, and I. Siap, “Codes over $\mathbb{F}_4 + v\mathbb{F}_4$ and some DNA applications,” Designs, Codes and Cryptography, vol. 80, no. 2, pp. 379–393, 2016.

    \bibitem{NabilBennenni} N. Bennenni, K. Guenda, and S. Mesnager, “Dna cyclic codes over rings,” 2017.
    
    \bibitem{dinh2018cyclic} H. Q. Dinh, A. K. Singh, S. Pattanayak, and S. Sriboonchitta, “Cyclic DNA codes over the ring $\mathbb{F}_2 +u\mathbb{F}_2+v\mathbb{F}_2 +uv\mathbb{F}_2 +v^2\mathbb{F}_2 +uv^2\mathbb{F}_2$,” Designs, Codes and Cryptography, vol. 86, no. 7, pp. 1451–1467, 2018.
    
    \bibitem{dinh2018dna} H. Q. Dinh, A. K. Singh, S. Pattanayak, and S. Sriboonchitta, “DNA cyclic codes over the ring $\mathbb{F}_2[u, v]/\langle u^2 -1, v^3 - v, uv - vu\rangle$,” International Journal of Biomathematics, vol. 11, no. 03, p. 1850042, 2018.

    \bibitem{liu2020dna} J. Liu and H. Liu, “Dna codes over the ring $\mathbb{F}_4[u]/\langle u^3\rangle$,” IEEE Access, vol. 8, pp. 77528–77534, 2020.

    \bibitem{kanlaya2023constructing} A. Kanlaya and C. Klin-Eam, “Constructing double cyclic codes over $\mathbb{F}_2 + u\mathbb{F}_2$ for DNA codes,” Journal of Computational Biology, vol. 1, no. 1, pp. 1112–1130, 2023.

    \bibitem{kanlaya2025constructing} A. Kanlaya and C. Klin-Eam, “Constructing DNA codes using double cyclic codes of odd and even lengths over $\mathbb{F}_2 + u\mathbb{F}_2$,” Carpathian Journal of Mathematics, vol. 41, no. 1, pp. 137–156, 2025.

    \bibitem{borges2018z} J. Borges, C. Fern´andez-C´ordoba, and R. Ten-Valls, “Z2-double cyclic codes,” Designs, Codes and Cryptography, vol. 86, no. 3, pp. 463–479, 2018.
    
    \bibitem{patanker2023reversible} N. Patanker, “On reversible $\mathbb{Z}_2$-double cyclic codes,” Bulletin of the Korean Mathematical Society, vol. 60, no. 2, pp. 443–460, 2023.

    \bibitem{oztas2013lifted} E. S. Oztas and I. Siap, “Lifted polynomials over $F_{16}$ and their applications to DNA codes,” Filomat, vol. 27, no. 3, pp. 459–466, 2013.

    \bibitem{huffman2010fundamentals} W. C. Huffman and V. Pless, Fundamentals of error-correcting codes. Cambridge University Press, 2010.

    \end{thebibliography}
    
\end{document}